\documentclass[12pt]{amsart}

\usepackage{amsmath}
\usepackage{amsfonts}
\usepackage{amssymb}
\usepackage{epsfig}
\usepackage{epstopdf}
\usepackage{color}

\usepackage{cancel}
\usepackage[margin=1in]{geometry}

\newcommand*{\mathtext}[1]{%
  \expandafter\def\csname#1\endcsname{\operatorname{#1}}}
\mathtext{rank}

\usepackage[pdftex]{hyperref}

\newtheorem{theorem}{Theorem}[section]
\newtheorem{lemma}[theorem]{Lemma}

\theoremstyle{definition}

\newcommand{\CC}{\mathbb C}
\newcommand{\RR}{\mathbb R}
\newcommand{\QQ}{\mathbb Q}
\newcommand{\ZZ}{\mathbb Z}

\date{\today}
 
\title[]{Toric h-vectors and Chow Betti Numbers of Dual Hypersimplices}
\author{Charles Wang}
\address{Department of Mathematics, UC Berkeley}
\email {charles.m.wang@berkeley.edu}

\author{Josephine Yu}
\address{School of Mathematics, Georgia Tech, Atlanta GA, USA}
\email{jyu@math.gatech.edu}

\thanks {}

\date{\today}

\begin{document}

\begin{abstract}
The toric $h$-numbers of a dual hypersimplex and the Chow Betti numbers of the normal fan of a hypersimplex are the ranks of intersection cohomology and Chow cohomology respectively of the torus orbit closure of a generic point in the Grassmannian.  We give explicit formulas for these numbers.  We also show that similar formulas hold for the coordinator numbers of type $A^*$ lattices.
\end{abstract}

\maketitle
 
\section{Introduction}

The hypersimplex $\Delta_{k,n}$ is the convex hull of all $0/1$-vectors in $\RR^n$ with exactly $k$ $1$'s. The $n-1$ dimensional complex projective toric variety, $X_{k,n}$, corresponding to the normal fan of $\Delta_{k,n}$ is the closure of the torus orbit of a generic point in the Grassmannian $\mathop{Gr}_{k,n}(\CC)$. Furthermore, $X_{k,n}$ is singular unless $k=1$ or $k=n-1$, when $\Delta_{k,n}$ is a simplex.

For smooth toric varieties associated to the normal fan of a simplicial polytope, the ranks of cohomology groups of the toric variety are the $h$-numbers of the polytope, which depend only on the face numbers of the polytope.  For singular toric varieties, which correspond to non-simplicial polytopes, Stanley introduced the {\em toric $h$-numbers}, which are the ranks of {\em intersection cohomology} groups of the toric variety~\cite{Stanley85}.  These numbers can also be computed combinatorially from the face lattice of the polytope.  We give a new proof of the formula, originally due to Jojic~\cite{Jojic}, for toric $h$-vectors of dual hypersimplices.

Fulton and Sturmfels showed that the rational Chow cohomology ring of a complete toric variety coincides with the space of Minkowski weights, which are functions on the set of cones of the corresponding fan that satisfy certain balancing conditions~\cite{FultonSturmfels}.  We prove a formula for the Chow Betti numbers of normal fans of hypersimplices by giving a basis for the Minkowski weight space.

The $d$-dimensional diplo-simplex is the convex hull of the shortest vectors in the lattice~$A_d^*$.  Its Ehrhart series counts the lattice points of a given distance from the origin and is called the {\em growth series} of the lattice.  We prove the formula for this growth series stated by Conway and Sloane~\cite{ConwaySloane} based on the work of O'Keeffe~\cite{OKeeffe}.  The proof is closely related to the toric $h$ numbers of dual hypersimplices.

\section{Toric h-numbers of $\Delta(k,n)^*$}

In \cite{Stanley85}, Stanley defined the toric $h$-numbers of a polytope via its face poset. They coincide with the ranks of (even dimension) intersection cohomology groups of the toric variety of the polytope's normal fan.

Two polynomials, $h(L,x)$ and $g(L,x)$, are defined inductively on Eulerian posets $L$ as follows\footnote{The polynomial $h$ corresponds to $f$ in~\cite{ec1,Stanley85} and $h$ in~\cite{kalai}}.  
For each Eulerian poset $L$, let $\widehat{0}$ and $\widehat{1}$ denote the least and greatest elements of $L$ respectively, $[\widehat{0},p] = \{q \in L : q \leq p\}$ for $p \in L$, and $\rank(p) = \rank[\widehat{0},p]$ where the rank of a graded poset is one less than the number of elements in any maximal chain. Then
\begin{itemize}
\item $h({\mathbf 1}, x) = g({\mathbf 1}, x) = 1$, where ${\mathbf 1}$ denotes the one element poset
\item For $\rank(L) > 0$:
\begin{equation}
\label{eq:torich}
h(L,x)= \sum_{p \in L - \{\widehat{1}\}} g([\hat{0},p],x)(x-1)^{\rank \widehat{1}- \rank(p)  - 1}
\end{equation}
\item For $\rank(L) > 0$, $d = \rank(L) - 1$, and $h(P,x) = h_0 + h_1 x + \cdots + h_d x^d$: 
\begin{equation}
\label{eq:toricg} 
g(L,x) = h_0 + (h_1 - h_0) x + (h_2 - h_1) x^2 + \cdots + (h_m - h_{m-1}) x^m
\end{equation} where $m = \lfloor d/2 \rfloor$. 
\end{itemize}
See~\cite[\S3.16]{ec1} for explicit examples.

For a polytope $P$ with face lattice $L$ and $h$ polynomial $h(L,x) = h_0 + h_1 x_1 + \cdots + h_d x^d$, the sequence of coefficients $(h_0, h_1, \dots, h_d)$ of $h$ is called the {\em toric $h$-vector} of $P$, where $d = \rank(L) - 1 = \dim(P)$.  It is symmetric ($h_i=h_{d-i}$ for all $0\le i\le d$) and unimodal ($1 = h_0 \leq h_1 \leq \cdots \leq h_{\lfloor (d-1)/2 \rfloor}$)~\cite{Stanley85}.  For a simplicial polytope $P$, the toric $h$-vector coincides with the usual $h$-vector defined by the coefficients of
\begin{equation}
\label{eq:usualh}
\sum_{\text{face } F \subseteq P} (x-1)^{d - \dim F} = h_0 x^d + h_1 x^{d-1} + \cdots + h_d.
\end{equation}

A $d$-dimensional polytope is called {\em quasi-simplicial} or {\em $d-2$ simplicial} if all of its faces of dimension $\leq d-2$ are simplices. Equivalently, all its facets are simplicial polytopes.
For a quasi-simplicial polytope $P$, the first half of the toric $h$-vector coincides with the first half of the usual $h$-vector, which follows from comparing \eqref{eq:torich} and \eqref{eq:usualh} and using that the $g$ polynomial for simplices is the constant polynomial~$1$.

\begin{lemma}
\label{quasi-simp}
For $0 \leq k \leq n$, the polytope $\Delta_{k, n}^*$ is quasi-simplicial. 
\end{lemma}
\begin{proof}
We will show that $\Delta_{k,n}$ is quasi-simple, i.e.\ every edge is contained in exactly $d-1$ $2$-faces.  By \cite{FultonSturmfels}, every face of $\Delta_{k,n}$ is of the form $\mathcal{F}_{I,J}$, where $|I|<k$ and $|J|<n-k$. In particular, an edge of $\Delta_{k,n}$ has the form $\mathcal{F}_{I,J}$ where $I \cap J = \varnothing$, $|I|=k-1$ and $|J|=n-k-1$. 

For an edge $\mathcal{F}_{I,J}$, any $2$-face containing the edge has the form $\mathcal{F}_{I',J'}$ where $I'\subset I$, $J'\subset J$, and $|I'|+|J'| = n-3$, so it is obtained by deleting a single element from either $I$ or $J$. Since $|I|+|J|=n-2$, there are exactly $n-2$ $2$-faces containing each edge of $\Delta_{k,n}$. 
\end{proof}

Let the $r^{th}$ entry of the toric $h$-vector of $P$ be denoted by $h_r(P)$.  The following follows from~\cite[Theorem~5]{Jojic}.  We give a new proof here.  See Table~\ref{table:compare} on \pageref{table:compare} for examples.
\begin{theorem}[Toric $h$ numbers]
\label{thm:toricH}
For any $1 \leq k \leq \lfloor n/2 \rfloor$,
\[
h_r(\Delta_{k,n}^*) = \left\{ 
\begin{array}{cl}
\sum\limits_{i=0}^r {n \choose i}  & \text{ if } 0 \leq r \leq k-1 \\
\sum\limits_{i=0}^{k-1} {n \choose i} & \text{ if } k \leq r \leq \lfloor \frac{n}{2} \rfloor
\end{array}
\right. .
\] 
\end{theorem}


\begin{proof}

  Let $0\le r\le k-1$. Then $\Delta_{k,n}^*$ has $f_r = \sum_{i=0}^{r+1} \binom{n}{i}\binom{n-i}{r+1-i}=2^{r+1}\binom{n}{r+1}$ faces of dimension~$r$. Let $f(x)=x^{n-1}+\sum_{r=0}^{n-2} f_r\, x^{n-2-r}$. Since $\Delta_{k,n}^*$ is quasi-simplicial, it suffices to compute the $r^{th}$ entry of the usual $h$-vector by the coefficient of $x^{n-1-r}$ in $f(x-1)$. Since $f$ coincides with $\frac{(x+2)^n-2^n}{x}$ for the $k$ highest degree terms, the coefficients of $x^{n-1-r}$ in $f$ and in $\frac{(x+1)^n-2^n}{x-1}$ coincide. It is straightforward to check that $(x-1)\sum_{i=0}^n \left(\sum_{j=0}^i \binom{n}{j}\right) x^{n-1-i}$ and $(x+1)^n-2^n$ coincide for the $k$ highest degree terms, and this gives the desired toric $h$-numbers for $0 \leq r \leq k-1$. 

Now suppose $k \leq r \leq \lfloor \frac{n}{2} \rfloor$. The sum of toric $h$-numbers is the value of the $h$-polynomial at $x=1$.  By~\eqref{eq:torich}, the only nonzero contributions to $h$ at $x=1$ come from corank one elements. Evaluating the $g$-polynomial in~\eqref{eq:toricg} at $x=1$ gives the middle coefficient of the $h$-polynomials.  Thus the sum of the toric $h$-numbers of a polytope $P$ is the sum of the middle $h$-numbers of the facets of $P$.
Each of the ${n \choose k}$ facets of $\Delta_{k,n}^*$ is isomorphic to the free sum $\Delta_{k-1} \oplus \Delta_{n-k-1}$. 
The $h$-polynomial of a free sum is the product of $h$-polynomials~\cite[Remark~3, p.205]{kalai}, so each $\Delta_{k-1} \oplus \Delta_{n-k-1}$ has $h$-polynomial $(x^{k-1}+\cdots+x+1)\cdot(x^{n-k-1}+\cdots+x+1)$, which has middle coefficient $k$. Thus the sum of toric $h$-numbers of $\Delta_{k,n}^*$ is equal to $k \cdot {n \choose k}$.

Since the toric $h$-vector is unimodal and $h_{k-1} = h_{n-k} = \sum_{i=0}^{k-1}\binom{n}{i}$ from above, we have
\begin{equation}
\label{eq:middlesum}
h_k +\cdots + h_{n-k-1} \geq (n-2k)\sum_{i=0}^{k-1}\binom{n}{i}
\end{equation}
with equality if and only if $h_k=\dots=h_{n-k-1}=\sum_{i=0}^{k-1}\binom{n}{i}$.  On the other hand, 
\[
k\binom{n}{k}
=(n-2(k-1))\binom{n}{k-1}+(k-1)\binom{n}{k-1} = \cdots =\sum_{i=0}^{k-1}(n-2i)\binom{n}{i}
\]
and $\sum_{i=0}^{k-1}\sum_{j=0}^i\binom{n}{j}  = \sum_{i=0}^{k-1}(k-i)\binom{n}{i}$, so
\[
h_k +\cdots + h_{n-k-1} = k\binom{n}{k}-2\sum_{i=0}^{k-1}\sum_{j=0}^i\binom{n}{j}=(n-2k)\sum_{i=0}^{k-1}\binom{n}{i}.
\]
Thus the inequality in \eqref{eq:middlesum} is tight, and $h_k=\dots=h_{n-k-1}=\sum_{i=0}^{k-1}\binom{n}{i}$.
\end{proof}

\section{Chow-Betti numbers}

The {\em Chow-Betti number}, $\beta_{r,k,n}$, of the hypersimplex $\Delta(k,n)$ is the rank of the rational Chow cohomology group, $A^r(X_{k,n})$, of $X_{k,n}$, the toric variety associated to the normal fan of $\Delta_{k,n}$. The group $A^r(X_{k,n})$ is free abelian and is naturally isomorphic to the group of rational valued Minkowski weights on the codimension-$r$ cones in the normal fan of $\Delta_{k,n}$~\cite{FultonSturmfels}.  Fulton and Sturmfels computed $\beta_{r,k,n}$ for small values of $k$ and $n$ and asked whether there exist nice general formulas.  We give such formulas in Theorem~\ref{thm:ChowBetti}.

The codimension-$r$ cones in the normal fan of $\Delta_{k,n}$ are indexed by pairs $(I,J)$ where $I$ and $J$ are disjoint subsets of $\{1,2,\dots,n\}$ with 
\begin{equation}
\label{eq:sizes}|I|+|J| = n-r-1,~~ |I| < k, \text{ and } |J| < n-k.
\end{equation}
We define the {\em level}, $\ell^{r,k,n}(I,J)$, of a pair $(I,J)$ to be the difference between $|I|$ and its maximum possible value satisfying~\eqref{eq:sizes}.  For example, for $n=8, k=3, r=2$, we have $|I|+|J| = 8-2-1 = 5$, and $|I| < 3, |J| < 8-3 = 5$, so $1 \leq |I| \leq 2$.  Thus $(\{1,2\},\{3,4,5\})$ is at level $0$ while $(\{1\},\{2,3,4,5\})$ is at level $1$.

It follows from the description of the faces and the normal fan of $\Delta_{k,n}$ that a Minkowski weight on codimension $r$ cones is a functions~$c$ on the set of ordered pairs as in~\eqref{eq:sizes} satisfying the following {\em balancing conditions}~\cite{FultonSturmfels}
\begin{align}
\label{eq:middle} c(Ax,B) - c(Ay,B)  &= c(A,Bx) - c(A,By) && \text{ if } |A| < k-1 \text{ and } |B| < n-k-1 \\
\label{eq:Imax} c(I,J)  &= c(I, J') && \text{ if } |I| = k-1 \\
\label{eq:Jmax} c(I,J)  &= c(I', J) && \text{ if } |J| = n-k-1 
\end{align}
where $Ax$ denotes $A \cup \{ x \}$.

The following is our main result about Chow-Betti numbers. We only consider $1\le k\le \lfloor\frac{n}{2}\rfloor$ because $\Delta_{k,n}$ and $\Delta_{n-k,n}$ are isomorphic. We have $\beta_{0,k,n} = 1$ for all $k$ and $n$ because constants are the only Minkowski weights on the full dimensional cones of a complete fan.

\begin{theorem}[Chow-Betti numbers]
\label{thm:ChowBetti} 
For $1\le k\le \lfloor\frac{n}{2}\rfloor$, 
\[
\beta_{r,k,n}= \left\{ 
\renewcommand\arraystretch{2}
\begin{array}{cl}
\sum\limits_{i=0}^{r-1} {n \choose i} & \text{ if } 1 \leq r \leq k\\
\sum\limits_{i=0}^{k-1} {n \choose i} & \text{ if } k < r < n- k\\
\sum\limits_{i=0}^{n-r-1} {n \choose i} & \text{ if } n-k \leq r \leq n-1
\end{array}
\right. .
\]
A basis for the Chow cohomology $A^r(X_{k,n})$ over $\QQ$ is given by the following Minkowski weights $c^{r,k,n}_S$ defined as follows for $(I,J)$ satisfying~\eqref{eq:sizes}
\[c^{r,k,n}_S(I,J)=\begin{cases} 1 & \text{if } S \subseteq I \cup J \text{ and } |J \cap S|= \ell^{r,k,n}(I,J) \\ 0& \text{otherwise}  \end{cases}
\]
where  $S$ runs over all subsets of~$\{1,2,\dots,n\}$ of size at most $r-1$, $k-1$, and $n-r-1$ respectively for the three cases.

\end{theorem}
\noindent Note that the number of levels for the $(I,J)$'s that appear in these three cases are $r$, $k$, and $n-r$ respectively.

\begin{proof}

We first show that the $c^{r,k,n}_S$ are Minkowski weights, i.e.\ they satisfy \eqref{eq:middle}, \eqref{eq:Imax}, and \eqref{eq:Jmax} above. We often write $c_S$ instead of $c^{r,k,n}_S$ to simplify notation.

Condition~\eqref{eq:Imax} only applies when $|I| = k-1$ at level $0$.  Then $c_S(I,J) = 1$ iff $I \supseteq S$, so \eqref{eq:Imax} holds. Similarly, condition~\eqref{eq:Jmax} only applies when $|J| = n-k-1$ at highest level. If $|S|$ is less than the number of levels, then $c_S(I,J)=0$ for all $(I,J)$.  If $|S|$ is equal to the number of levels, then $c_S$ takes $0/1$ values with $c_S(I,J) = 1$ iff $J \supseteq S$.  In either case, \eqref{eq:Jmax} is satisfied.

Finally, we look at cases where condition~\eqref{eq:middle} applies. Let $A,B,x,y$ be disjoint with $|A| < k-1$, $|B| < n-k-1$, and $x,y$ are singletons. Then we have the following cases:
\begin{itemize}
\item ($c_S(Ax,B) = 1$ and $x \in S$) Then $c_S(A,Bx) = 1$; $c_S(Ay,B) = c_S(A,By) = 0$.
\item ($c_S(Ax,B) = 1$ and $x \notin S$) Then necessarily $y \notin S$, so $c_S(Ay,B) = 1$ and  $c_S(A,Bx) = c_S(A,By) = 0$.
\item ($c_S(Ax,B) = 0$ and $x \in S$) Then $c_S(A,Bx) = c_S(Ay,B) = c_S(A,By) = 0$.
\item ($c_S(Ax,B) = 0$ and $x \notin S$)  If $c_S(Ay,B) = 1$, then this is covered by the first two cases above by interchanging $x$ and $y$, so we may assume that $c_S(Ay,B) = 0$.\\
If $y \in S$, then $c_S(A,Bx) = 0$ and $c_S(A,By) = 0$.\\
If $y \notin S$, then $c_S(A,Bx) = c_S(A,By)$.
\end{itemize}
In each of the cases above, \eqref{eq:middle} is satisfied.  Thus, the functions $c_S$ are Minkowski weights.

\smallskip

We now show that the $c_S$ are linearly independent.  We will use the following result of de~Caen on {\em set-inclusion matrices}: Let $W_{i,j}(n)$ be the $0/1$ matrix whose rows and columns are indexed by subsets of $\{1,2,\dots,n\}$ of size $i$ and $j$ respectively where $W_{A,B}=1$ if $A\subseteq B$ and $0$ otherwise. Then the rank of $W_{i,j}(n)$ is $\binom{n}{i}$ when $i \leq j$ and $i+j\le n$~\cite{deCaen}.

Fix $r,k$, and $n$. Let $M$ be a matrix whose columns are indexed by pairs $(I,J)$ satisfying~\eqref{eq:sizes} and rows by subsets $S\subset \{1,2,\dots,n\}$ of size less than the number of levels, with entry $M_{S,(I,J)}=c_S(I,J)$.  We order the sets $S$ from smallest to largest and the pairs $(I,J)$ by the size of $J$.  Then the matrix $M$ is block lower triangular where the $i^{th}$ diagonal block ($i\ge 0$) has rows indexed by sets $S$ of size $i$ and the columns by pairs $(I,J)$ at level $i$.  In the $i^{th}$ block, $M_{S,(I,J)} = 1$ if and only if $S \subseteq J$, so the columns of the $i^{th}$ blocks are exactly the same (with repetition) as the columns of the set-inclusion matrix $W_{i,j}(n)$ where $j$ is the size of $J$ at level $i$.  To complete the proof that the $i^{th}$ block has rank ${n \choose i}$, which is the number of rows, we check that $i+j \leq n$ in each of the following cases.  Recall that we are assuming $k \leq n-k$.
\begin{itemize}
\item ($1 \leq r < n-k$) The maximum $|I|$ is $k-1$, so the minimum $|J|$ is $n-r-k$.  For $(I,J)$ at level $i$, we have $j = |J| = n-r-k+i$.
\begin{itemize}
\item ($1 \leq r \leq k$) The number of levels is $r$, so $i < r$, and $i+j = n - r - k + 2i < n-r-k+2r = n - (k-r) \leq n$.
\item ($k < r < n-k$) The number of levels is $k$, so $i < k$, and $i+j = n-r-k+2i < n-r-k+2k = n - (r-k) < n$.
\end{itemize}
\item ($n-k \leq r \leq n-1$)  For $(I,J)$ at level $i$, we have $j = |J|= i$.  The number of levels is $n-r$, so $i < n-r \leq n/2$.  Thus $i+j = 2i < n$.
\end{itemize}
This proves that the $c_S$ are linearly independent.

\smallskip

We will now show that the functions $c_S$ span Minkowski weight space. First consider when $n-k\le r\le n-1$.  Then only condition~\eqref{eq:middle} applies.  We will choose values for a Minkowski weight on the pairs $(I,J)$ one level at a time.  At level $0$, $|J| = 0$ and $|I| = n-r-1$, and we choose the values on each of these pairs $(I,\varnothing)$ independently, giving us ${n \choose n-r-1}$ degrees of freedom.  After we have chosen values at level $i$, the equations \eqref{eq:middle} with LHS at level $i$ (with $|A| = n-r-1-i$) leave one degree of freedom at level $i+1$ for each $A$.  This gives ${n \choose n-r-1-i}$ degrees of freedom.  Continuing in this fashion, we get to level $n-r-1$ where $I = \varnothing$, with $1$ degree of freedom. Hence the dimension of Minknowski weight space is at most $\sum_{i=0}^{n-r-1} {n \choose i}$, which coincides with the lower bound proven above.

\smallskip

Now consider the case when $k < r < n-k$.  In this case the conditions \eqref{eq:middle} and \eqref{eq:Imax} both apply, but \eqref{eq:Jmax} does not.  The level $0$ has $|I| = k-1$, so the condition \eqref{eq:Imax} leaves ${n \choose k-1}$ degrees of freedom at level $0$.  Using condition \eqref{eq:middle}, the next levels have degrees of freedom equal to ${n \choose k-2}$, ${n \choose k-3}, \dots,$ respectively up to ${n \choose 0}=1$ when $I = \varnothing$.  This gives the desired upper bound on dimension, which is $\sum_{i=0}^{k-1}\binom{n}{i}$.

\smallskip

Finally, let $1\le r\le k$. In this case all three types of conditions \eqref{eq:middle}, \eqref{eq:Imax}, and \eqref{eq:Jmax} apply. The same approach as above does not give a tight upper bound since it does not use condition~\eqref{eq:Jmax}.

 For $n \leq 3$ it is easy to check that the functions $c^{r,k,n}_S$ span Minkowski weight space for all possible values of $r$ and $k$.  Let $n \geq 4$ and fix $1 \leq r \leq k$.  The number of levels is $r$.
Let $c$ be a Minkowski weight on $\Delta_{k,n}$. 
Define $c'$ on pairs $(I,J)$ where $I,J\subset \{1,\dots, n-1\}$, $|I|+|J|=(n-1)-r-1$, $|I|<k$, and $|J|<n-k$ by $c'(I,J)=c(I,Jn)$. It is straightforward to check that $c'$ is a Minkowski weight on codimension $r$ cones of $\Delta_{k,n-1}$.  
For $S \subseteq [n-1]$ with $|S| \leq r-1$, the function $c^{r,k,n-1}_S$ is obtained from $c^{r,k,n}_S$ in this way. Note that the number of levels is still $r$ for $r,k,n-1$, so $c_S^{r,k,n-1}$ is defined.  
By induction on $n$ the functions $c^{r,k,n-1}_S$ span the Minkowski weight space on codimension $r$ cones for $\Delta_{k,n-1}$, and only the $c_S$ with $|S| = r-1$ have non-zero coordinates at the top level, so there are coefficients $a_S$ such that $c'(I,J)=\sum_{S\in {[n-1] \choose r-1}} a_S c^{r,k,n-1}_S(I,J)$ for all $(I,J)$ at top level for $r,k,n-1$.

Let $c_1 = c-\sum_{S \in {[n-1] \choose r-1}} a_S c^{r,k,n}_S$.  It is a Minkowski weight for $r,k,n$ with $c_1(I,J) = 0$ for all top level coordinates $(I,J)$ with $n \in J$ or $n\notin I \cup J$.
Define a Minkowski weight $c_1''$ on  codimension $r$ cones of $\Delta_{k-1,n-1}$ by $c_1''(I,J)=c_1(In, J)$. For $|S|=r-1$ and $n \in S$ we have $(c^{r,k,n}_S)'' = c^{r,k-1,n-1}_{S-\{n\}}$.  Note that the number of levels for $r, k-1, n-1$ is $r$ when $r<k$ and $k-1$ when $r=k$; in either case $|S-\{n\}|=r-2$ is less than the number of levels, so $c^{r,k-1,n-1}_{S-\{n\}}$ is defined.   We can write $c_1''=\sum_{S \in {[n-1] \choose r-2}} a_S c^{r,k-1,n-1}_{S}$. Let $c_2=c_1-\sum_{S \in {[n-1] \choose r-2}} a_S c^{r,k,n}_{Sn}$. By construction, $c_2(I,J) = 0$ on all coordinates $(I,J)$ at top level with $n \in I$. Since all these $c_S$ are also $0$ on coordinates with $n \notin I$, then $c_2$ is a Minkowski weight that is $0$ on all $(I,J)$ at the top level. 

The condition~\eqref{eq:Jmax} on $c_2$ becomes, when $|I|=k-2$, $0=c_2(I,Jx)-c_2(I,Jy)$. This allows us to apply the same argument as above, using the weights $c_S$ where $|S|=r-2$, to get a Minkowski weight which is $0$ on all $(I,J)$ with $I\ge k-2$. We can repeat this process until we have $|S|=r-1-(r-1)=0$. When we reach this case, we will have a Minkowski weight which is $0$ on all $(I,J)$ with $I\ge k-(r-1)$. By condition \eqref{eq:Jmax}, we get a multiple of the Minkowski weight $c_{\varnothing}$. Thus, we have written $c=\sum_{|S|<r} a_S c_S$ as a linear combination of the functions $c_S$, so the $c_S$ span the space of Minkowski weights. 
\end{proof}

\section{Coordinator numbers of $A_{n-1}^*$ lattices}

Let $n$ be an integer $\geq 3$. The type $A_{n-1}$ lattice is the sublattice of $\ZZ^n$ spanned by the type $A_{n-1}$ roots $\{e_i - e_j : i \neq j\}$.  The type $A_{n-1}^*$ lattice is the dual lattice to $A_{n-1}$, which we will identify with the image of the orthogonal projection of $\ZZ^n$ onto the hyperplane $\{x \in \RR^n : x_1 + \cdots + x_n = 0\}$.  

The shortest vectors in the $A_{n-1}^*$ lattice are the projection of $\{e_1, -e_1, \dots, e_n, -e_n\} \subset \ZZ^{n}$ along the direction $(1,1,\dots,1)$ onto the hyperplane above.  Any collection of $n-1$ linearly independent vectors among them form a $\ZZ$-basis of the lattice, that is, the shortest vectors form a unimodular system. Their convex hull is the $n-1$-dimensional {\em diplo-simplex} which we denote by $\Delta(\frac{n}{2},n)^*$. It is polar to the following slice of the $n$-cube $[-1,1]^n$
\[
\left\{x \in \RR^n : x_1 + \cdots + x_n = 0 \text{ and } -1 \leq x_i \leq 1 \text{ for all } i \right\}
\]
which we will denote by $\Delta(\frac{n}{2},n)$ by abuse of notation. When $n$ is even, this is a translate of the ``usual'' hypersimplex dilated by $2$. 

The {\em coordination sequence} of the type $A_{n-1}^*$ lattice with respect to the shortest vectors is the sequence $S(0), S(1), S(2), \dots$ where $S(k)$ is the number of lattice points on the boundary of $k \cdot \Delta(\frac{n}{2},n)^*$.  Its generating function $G(x) := \sum_{k\geq 0} S(k) x^k$ is called the {\em growth series}.  By Ehrhart theory $G(x) = \frac{h(x)}{(1-x)^n}$ where $h(x)$ is a polynomial of degree $< n$, called the {\em coordinator polynomial}.   The coefficients of $h(x) = h_0 + h_1 x + \cdots + h_{n-1} x^{n-1}$ are called the {\em coordinator numbers} of the $A_{n-1}^*$ lattice.  
 
Using results of O'Keeffe~\cite{OKeeffe},  Conway and Sloane computed the coordinator numbers of the lattice $A_{n-1}^*$  for small values of $n$~\cite{ConwaySloane}\footnote{See also the OEIS sequence A204621 at \url{http://oeis.org/A204621}.}.  Here we prove the general formula.
 
\begin{theorem}[Coordinator numbers]
For any positive integer $n \geq 2$, the $r^{th}$ coordinator number of $A_{n-1}^*$ is equal to
$\sum_{i=0}^r {n \choose i}$ if $0\leq r \leq \frac{n-1}{2}$  and 
$\sum_{i=0}^{n-1-r} {n \choose i}$ if $ \frac{n-1}{2} \leq r \leq n-1$.
\end{theorem}

 For a $d$-dimensional simplical complex $\Delta$ in $\RR^n$, the $f$-polynomial is $f(x) = \sum_{\sigma \in \Delta} x^{d - \dim(\sigma)}$  and the $h$-polynomial is  $h(x) = f(x-1)$.   
The Ehrhart $h^*$-polynomial of $\Delta$ is  \[h^*(x) = (1-x)^{d+1} ( 1 + \sum_{k > 1} |k \Delta \cap \ZZ^m| x^k).\]  A $j$-dimensional unimodular simplex in $\Delta$ contributes $x^{j+1}(x-1)^{d-j}$ to the $h$-polynomial and $(1-x)^{d-j}$ to the $h^*$-polynomial. If $\Delta$ consists only of unimodular simplices, then  by inclusion-exlusion, we have 
\[
h(x) = x^{d+1} \cdot h^*(\frac{1}{x}).
\]
Also see~\cite[Theorem~3]{ABHPS} and references therein.

The numerator of the growth series is the Ehrhart $h^*$-polynomial of the boundary of the diplo-simplex.  By the discussion above it suffices to compute the $h$-polynomial of a triangulation of the boundary of the diplo-simplex, which is palindromic.

\begin{proof}
First consider the case when $n$ is even.  The diplo-simplex $\Delta\left(\frac{n}{2},n\right)^*$ is quasi-simplicial by  Lemma~\ref{quasi-simp}.  
Let $\Delta$ be the triangulation of the boundary of $\Delta(\frac{n}{2},n)^*$.  
The vertices form a unimodular system, so $\Delta$ is unimodular.  Moreover the triangulation does not introduce new faces of dimension $\leq \frac{n}{2}$ since all faces of codimension $\geq 2$ are already simplicial.

The $h$-vector of $\Delta$ is palindromic and determined by $f_0,f_1,\dots,f_{\lfloor \frac{n}{2} \rfloor}$.  Since the toric $h$-vector of $\Delta(\frac{n}{2},n)^*$ is also determined by the first half of the $f$-vector in the same way, the $h$-vector of $\Delta$ coincides with the toric $h$-vector of the polytope $\Delta(\frac{n}{2},n)^*$, hence we have the same formula proven earlier in Theorem~\ref{thm:toricH}.  On the other hand the coordinator numbers coincide with the $h$-numbers as explained above.

Now suppose $n$ is odd. Then no vertex of the $n$-cube $[-1,1]^n$  lies on the hyperplane $x_1 + \cdots + x_n = 0$.  Each face of the $n$-cube either intersects the hyperplane at a relative interior point or does not intersect at all.  The vertices of $\Delta\left(\frac{n}{2},n\right)$ are intersections of the interiors of some edges of the $n$-cube with a hyperplane.  Since the $n$-cube is simple, each edge is contained in $n-1$ facets.  Then each vertex of $\Delta\left(\frac{n}{2},n\right)$ is contained in $n-1$ facets as well.  Thus $\Delta\left(\frac{n}{2},n\right)$ is simple, and its dual $\Delta\left(\frac{n}{2},n\right)^*$ is simplicial.

The hyperplane  $x_1+\cdots+x_n = 0$ meets the relative interiors of all faces of dimension $\leq \lceil \frac{n}{2} \rceil$  of the $n$-cube.  Then the diplo-simplex $\Delta(\frac{n}{2},n)^*$ has $2^{i+1} {n \choose i+1}$ faces of dimension $i$ for each $0 \leq i \leq \lfloor \frac{n}{2} \rfloor$. By the same computation as in Theorem~\ref{thm:toricH} (for the case when $0\le r\le k-1$), the $r^{th}$ $h$-number of $\Delta(\frac{n}{2},n)$ is $\sum_{i=0}^r \binom{n}{i}$ for $0\le r\le \lfloor \frac{n}{2}\rfloor$. 
\end{proof}

\section{Open problems} 

For $\Delta(k,n)^*$ the toric $h$ numbers and Chow-Betti numbers of the normal fan have very similar forms. See Table~\ref{table:compare}. However our proofs are entirely independent from each other, and we do not know of a direct relationship between them.  

\begin{table}[!htbp]
{\footnotesize
\begin{tabular}{l|l|l}
$k,n$ & toric-$h$ vector & Chow--Betti numbers\\
\hline
2, 4 & 
1 5 5 1 &
1 1 5 1
\\
\hline
2, 5 & 
1 6 6 6 1 &
1 1 6 6 1
\\
\hline
2, 6 &
1 7 7 7 7 1 &
1 1 7 7 7 1
\\
3, 6 &
1 7 22 22 7 1 &
1 1 7 22 7 1
\\
\hline
2, 7 &
1 8 8 8 8 8 1 &
1 1 8 8 8 8 1
\\
3, 7 &
1 8 29 29 29 8 1 &
1 1 8 29 29 8 1
\\
\hline
2, 8 &
1 9 9 9 9 9 9 1 &
1 1 9 9 9 9 9 1
\\
3, 8 &
1 9 37 37 37 37 9 1&
1 1 9 37 37 37 9 1
\\
4, 8 &
1 9 37 93 93 37 9 1 &
1 1 9 37 93 37 9 1
\\
\hline
2, 9 &
1 10 10 10 10 10 10 10 1&
1 1 10 10 10 10 10 10 1
\\
3, 9 &
1 10 46 46 46 46 46 10 1&
1 1 10 46 46 46 46 10 1
\\
4, 9 &
1 10 46 130 130 130 46 10 1&
1 1 10 46 130 130 46 10 1
\\
\hline
2, 10 &
1 11 11 11 11 11 11 11 11 1&
1 1 11 11 11 11 11 11 11 1
\\
3, 10 &
1 11 56 56 56 56 56 56 11 1&
1 1 11 56 56 56 56 56 11 1
\\
4, 10 &
1 11 56 176 176 176 176 56 11 1&
1 1 11 56 176 176 176 56 11 1
\\
5, 10 &
1 11 56 176 386 386 176 56 11 1&
1 1 11 56 176 386 176 56 11 1
\\
\hline
\end{tabular}
}
\caption{Toric $h$-vector of $\Delta(k,n)^*$ and Chow-Betti numbers of the normal fan of $\Delta(k,n)$ graded by codimension of cones.}
\label{table:compare}
\end{table}

A general questions is: how are toric-$h$ and Chow-Betti numbers related?  They may have little resemblance in general.  For instance the Chow-Betti numbers of the normal fan of a simplicial polytope depend only on the number of facets of the polytope --- they have the form $(1,n-d,1,\cdots,1)$ where $n$ is the number of facets and $d$ is the dimension of the polytope.  

McConnell showed that the rational homology of toric varieties is not a combinatorial invariant~\cite{McConnell89}. In fact, the same example shows that the Chow-Betti numbers are also not a combinatorial invariant. 
Using gfan, we found that the Chow-Betti numbers of the normal fan of $P_1$ and $P_2$ in~\cite[Example~1.4]{McConnell89} are $(1,4,9)$ and $(1,3,9)$ respectively.  The Chow-Betti numbers of their face fans also differ, being $(1,2,11)$ and $(1,1,11)$ respectively. The toric h-vector of both these polytopes are $(1,11,11,1)$.  

We do not know whether different geometric realizations of (dual) hypersimplices can give different Chow-Betti numbers. While we expect that such examples do exist, we have not been able to find any computationally. In general, for which combinatorial types of polytopes (or fans) do all geometric realizaions have the same Chow-Betti numbers? 

\section*{Acknowledgments}
We used Gfan~\cite{gfan}, Polymake~\cite{polymake}, and Macaulay2~\cite{GS} extensively for experiments.
CW was supported by a President's Undergraduate Research Award at Georgia Tech.  Both authors were partially supported by NSF-DMS grant \#1600569.

\bibliographystyle{amsalpha}
\bibliography{mybib}
 
 \end{document}